\documentclass[11pt]{article}
\usepackage{geometry}
\usepackage{graphicx}
\usepackage{amssymb}
\usepackage{epstopdf}
\usepackage{amsmath}
\usepackage{amsfonts,amsthm}
\usepackage{graphics}
\newtheorem{lemma}{Lemma}
\newtheorem{theorem}[lemma]{Theorem}

\newtheorem{conjecture}[lemma]{Conjecture}
\newtheorem{claim}[lemma]{Claim}

\title{Locating any two vertices on Hamiltonian cycles{\footnote {Supported by NSFC (Grant Nos. 11601093 and 11671296).}}}
\author{Weihua He $^{a,b}$,~ Hao Li $^{b,c}${\footnote {Corresponding author, E-mail:
li@lri.fr}},~ Qiang Sun $^{b}$\\ {\small $a.$ Department of Applied Mathematics,}\\ {\small Guangdong University of Technology, Guangzhou, China}\\ {\small $b.$ Laboratoire de Recherche
en Informatique,}\\ \small {UMR 8623, C.N.R.S.-Universit\'e Paris-sud, 91405 Orsay, France}\\
\mbox{}\hspace{0.15cm}\small {$c.$ Institute for Interdisciplinary Research,}\\
\small {Jianghan University, Wuhan, China}
}

\date{}

\begin{document}
\maketitle

\begin{abstract}
In this paper we give a proof of Enomoto's conjecture for graphs of sufficiently large order. Enomoto's conjecture states that, if $G$ is a graph of order $n$ with minimum degree $\delta(G)\geq \frac{n}{2}+1$, then for any pair of vertices $x$, $y$ in $G$, there is a Hamiltonian cycle $C$ of $G$ such that $d_C(x,y)=\lfloor \frac{n}{2}\rfloor$. The main tools of our proof are Regularity Lemma of Szemer\'{e}di and Blow-up Lemma of Koml\'{o}s et al.
\end{abstract}


\section{Introduction}

In this paper, a {\em graph} $G=(V(G),E(G))$ will be a finite undirected graph without loops or multiple edges. For any vertex $v$ of $G$ and a subset $X$ of $V(G)$, we denote the degree of $v$ in $G$ by $deg_G(v)$ and the degree of $v$ in $X$ by $deg_G(v,X)$ (if no ambiguity arises, we denote them by $deg(v)$ and $deg(v,X)$ respectively). A {\em Hamiltonian cycle} is a spanning cycle, i.e., the cycle visits each vertex of the graph exactly once. A graph is called {\em Hamiltonian} if it has a Hamiltonian cycle. For two vertices $u,v\in V(G)$, the {\em distance} $dist_G(u,v)$ is defined as the number of edges in a shortest path joining them in $G$. For terminology and notation not defined here, we refer to \cite{bondy}.

A graph is called a {\em Dirac graph} if the degree of every vertex is at least half of the order of the graph. In 1952, Dirac showed that the Dirac graph is Hamiltonian. Many results have been obtained in generalization of Dirac's theorem (see \cite{gouldsurvey}, \cite{dirac} for the surveys).

There are plenty of results to strengthen Dirac's theorem. One of the most interesting research area is to control the placement of a set of vertices on a Hamiltonian cycle such that these vertices have some certain distances among them on the Hamiltonian cycle. In 2001, Kaneko and Yoshimoto \cite{kaneko} showed that in a Dirac graph given any sufficiently small subset $S$ of vertices, there exists a Hamiltonian cycle $C$ such that the distances on $C$ between successive pairs of vertices of $S$ have a uniform lower bound.

\begin{theorem}\cite{kaneko}
Let $G$ be a graph of order $n$ with $\delta(G)\geq \frac{n}{2}$, and let $d$ be a positive integer such that $d\leq \frac{n}{4}$. Then, for any vertex subset $S$ with $|S|\leq \frac{n}{2d}$, there is a Hamiltonian cycle $C$ such that $dist_C(u,v)\geq d$ for any $u,v\in S$.
\end{theorem}

In 2008, S\'{a}rk\"{o}zy and Selkow \cite{sarkozy} showed that almost all of the distances between successive pairs of vertices of $S$ can be specified almost exactly.

\begin{theorem}\cite{sarkozy}\label{sarkozy}
There are $\omega , n_0>0$ such that if $G$ is a graph with $\delta(G)\geq \frac{n}{2}$ on $n\geq n_0$ vertices, $d$ is an arbitrary integer with $3\leq d\leq \frac{\omega n}{2}$ and $S$ is an arbitrary subset of $V(G)$ with $2\leq |S|=k\leq \frac{\omega n}{2}$, then for every sequence of integers with $3\leq d_i\leq d$, and $1\leq i\leq k-1$, there is a Hamiltonian cycle $C$ of $G$ and an ordering of the vertices of $S$, $a_1,a_2,...,a_k$, such that the vertices of $S$ are encountered in this order on $C$ and we have $|dist_C(a_i,a_{i+1})-d_i|\leq 1$, for all but one $1\leq i\leq k-1$.
\end{theorem}

In \cite{sarkozy}, the authors believe that Theorem \ref{sarkozy} remains true for greater values of $d$ as well. And many years ago Enomoto proposed the following conjecture of exact placement for a pair of vertices at a precise distance (half of the graph order) on a Hamiltonian cycle.

\begin{conjecture}\cite{gouldsurvey}\label{enomoto}
If $G$ is a graph of order $n\geq 3$ and $\delta(G)\geq \frac{n}{2}+1$, then for any pair of vertices $x$, $y$ in $G$, there is a Hamiltonian cycle $C$ of $G$ such that $dist_C(x,y)=\lfloor \frac{n}{2}\rfloor$.
\end{conjecture}

The degree condition of Enomoto's conjecture is sharp. First, we consider the complete bipartite graph $K_{\frac{n}{2},\frac{n}{2}}$. For any Hamiltonian cycle of $K_{\frac{n}{2},\frac{n}{2}}$, any pair of vertices in the same part will be at an even distance on this cycle and any pair of vertices in different parts will be at an odd distance on this cycle. Since $\delta(K_{\frac{n}{2},\frac{n}{2}})=\frac{n}{2}$, the minimum degree $\delta(G)\geq \frac{n}{2}$ is not sufficient to imply the existence of a Hamiltonian cycle with a fixed pair of vertices at distance $\lfloor \frac{n}{2}\rfloor$. Second, we consider the graph $(K_{\frac{n-3}{2}}\cup K_{\frac{n-3}{2}})+K_3$. If $x,y$ are both in one of the copies of $K_{\frac{n-3}{2}}$, then we cannot find a Hamiltonian cycle $C$ of $(K_{\frac{n-3}{2}}\cup K_{\frac{n-3}{2}})+K_3$ such that $dist_C(x,y)=\lfloor \frac{n}{2}\rfloor$. Since $\delta((K_{\frac{n-3}{2}}\cup K_{\frac{n-3}{2}})+K_3)=\frac{n+1}{2}$, the minimum degree $\delta(G)\geq \frac{n+1}{2}$ is not sufficient to imply the existence of the desired Hamiltonian cycle.

Motivated by Enomoto's conjecture, Faudree, Lehel, Yoshimoto \cite{faudree} and Faudree, Li \cite{li} deal with locating a pair of vertices at precise distances on a Hamiltonian cycle.

\begin{theorem}\cite{faudree}
Let $k\geq 2$ be a fixed positive integer. If $G$ is a graph of order $n\geq 6k$ and $\delta (G)\geq \frac{n}{2}+1$, then for any pair of vertices $x$, $y$ in $G$, there is a Hamiltonian cycle $C$ of $G$ such that $dist_C(x,y)=k$.
\end{theorem}

\begin{theorem}\cite{li}
If $k$ is a positive integer with $2\leq k\leq \frac{n}{2}$ and $G$ is a graph of order $n$ with  $\delta (G)\geq \frac{n+k}{2}$, then for any pair of vertices $x$ and $y$ in $G$, there is a Hamiltonian cycle $C$ of $G$ such that $dist_C(x,y)=p$ for any $2\leq p\leq k$.
\end{theorem}

Moreover, Faudree and Li \cite{li} proposed a more general conjecture.

\begin{conjecture}\cite{li}\label{li}
If $G$ is a graph of order $n\geq 3$ and $\delta(G)\geq \frac{n}{2}+1$, then for any pair of vertices $x$, $y$ in $G$ and any integer $2\leq k\leq \frac{n}{2}$, there is a Hamiltonian cycle $C$ of $G$ such that $dist_C(x,y)=k$.
\end{conjecture}

In this paper, we will prove Conjecture \ref{enomoto} for graphs of sufficiently large order. Our main result is the following.

\begin{theorem}\label{main}
There exists a positive integer $n_0$ such that for all $n\geq n_0$, if $G$ is a graph of order $n$ with $\delta(G)\geq \frac{n}{2}+1$, then for any pair of vertices $x$, $y$ in $G$, there is a Hamiltonian cycle $C$ of $G$ such that $dist_C(x,y)=\lfloor \frac{n}{2}\rfloor$.
\end{theorem}

\section{The main tools}

In this section we introduce some definitions about the regular pairs and some results related to these definitions.

Let $G$ be a graph, for any two disjoint vertex sets $X$ and $Y$ of $G$, the {\em density} of the pair $(X,Y)$ is the ratio $d(X,Y):=\frac{e(X,Y)}{|X||Y|}$, here $e(X,Y)$ is defined to be the number of edges with one end vertex in $X$ and the other in $Y$. Let $\epsilon >0$, we say the pair $(X,Y)$ is {\em $\epsilon $-regular} if for every $A\subseteq X$ and $B\subseteq Y$ such that $|A|>\epsilon |X|$ and $|B|>\epsilon |Y|$ we have $|d(A,B)-d(X,Y)|<\epsilon $. Moreover, let $\delta >0$, the pair $(X,Y)$ is called {\em ($\epsilon ,\delta$)-super-regular} if it is $\epsilon $-regular, $deg_Y(x)>\delta |Y|$ for all $x\in X$ and $deg_X(y)>\delta |X|$ for all $y\in Y$.

We will use some well-known properties of regular pairs.

\begin{lemma}\label{large}\cite{regularity}
Let $(A,B)$ be an $\epsilon $-regular pair of density $d$ and $Y\subseteq B$ such that $|Y|>\epsilon |B|$. Then all but at most $\epsilon |A|$ vertices in $A$ have more than $(d-\epsilon )|Y|$ neighbors in $Y$.
\end{lemma}

The following one says that subgraphs of regular pairs with reasonable size are also regular.

\begin{lemma}[Slicing Lemma]\label{slicing}\cite{regularity}
Let $\alpha >\epsilon >0$ and $\epsilon ^{'}:=\max \{\frac{\epsilon }{\alpha },2\epsilon \}$. Let $(A,B)$ be an $\epsilon $-regular pair with density $d$. Suppose $A^{'}\subseteq A$ such that $|A^{'}|\geq \alpha |A|$, and $B^{'}\subseteq B$ such that $|B^{'}|\geq \alpha |B|$. Then $(A^{'},B^{'})$ is an $\epsilon^{'}$-regular pair with density $d^{'}$ such that $|d^{'}-d|<\epsilon $.
\end{lemma}

For a bipartite graph $G=X\cup Y$, let $\delta(X,Y):=\min\{deg_G(x,Y):\ for\ every\ x\in X\}$. We can say that a bipartite graph with very large minimum degree has a super-regular pair.

\begin{lemma}\cite{chen}\label{chen}
Given $0<\rho<1$, let $G=X\cup Y$ be a bipartite graph such that $\delta (X,Y)\geq (1-\rho)|Y|$ and $\delta (Y,X)\geq (1-\rho)|X|$. Then $(X,Y)$ is ($\sqrt{\rho},1-\rho $)-super-regular.
\end{lemma}

Now we introduce Szemer\'{e}di's Regularity Lemma. We only state the degree form of the Regularity Lemma, which is more applicable (see \cite{regularity} for more details and applications of the Regularity Lemma).

\begin{lemma}[Regularity Lemma-Degree Form]\label{regular}
For every $\epsilon >0$ and every integer $m_0$ there is an $M_0=M_0(\epsilon, m_0)$ such that if $G=(V,E)$ is any graph on at least $M_0$ vertices and $d \in [0,1]$ is any real number, then there is a partition of the vertex set $V$ into $l+1$ clusters $V_0, V_1,..., V_l$, and there is a subgraph $G^{'}=(V,E^{'})$ with the following properties:\\
(1) $m_0\leq l\leq M_0$;\\
(2) $|V_0|\leq \epsilon |V|$, and $V_i$ $(1\leq i\leq l)$ are of the same size $L$;\\
(3) $deg_{G^{'}}(v)>deg_{G}(v)-(d +\epsilon)|V|$ for all $v\in V$;\\
(4) $G^{'}[V_i]= \emptyset $ (i.e. $V_i$ is an independant set in $G^{'}$) for all $i$;\\
(5) each pair $(V_i, V_j)$, $1\leq i<j\leq l$, is $\epsilon$-regular, each with a density $0$ or exceeding $d$.
\end{lemma}

An application of the Regularity Lemma in graph theory has a close relation with an application of the Blow-up Lemma. Here we only use the bipartite version of the Blow-up Lemma (see \cite{blowup} for the complete version).

\begin{lemma}[Blow-up Lemma-Bipartite Version]\label{blowup}
For every $\delta, \Delta, c>0$, there exists an $\epsilon =\epsilon(\delta, \Delta, c)>0$ and $\alpha =\alpha(\delta, \Delta, c)>0$ such that the following holds. Let $(X,Y)$ be an $(\epsilon, \delta)$-super-regular pair with $|X|=|Y|=N$. If a bipartite graph $H$ with $\Delta (H)\leq \Delta$ can be embedded in $K_{N,N}$ by a function $\phi$, then $H$ can be embedded  in $(X,Y)$. Moreover, in each $\phi^{-1}(X)$ and $\phi^{-1}(Y)$, fix at most $\alpha N$ special vertices $z$, each of which is equipped with a subset $S_z$ of $X$ or $Y$ of size at least $cN$. The embedding of $H$ into $(X,Y)$ exists even if we restrict the image of $z$ to be $S_z$ for all special vertices $z$.
\end{lemma}

Actually, we only need the following special case of the Blow-up Lemma in this paper.

\begin{lemma}\label{blow2}
For every $\delta >0$ there are $\epsilon_{BL}=\epsilon_{BL}(\delta)$, $n_{BL}=n_{BL}(\delta)>0$ such that if $\epsilon \leq \epsilon_{BL}$ and $N\geq n_{BL}$, $G=(X,Y)$ is an $(\epsilon, \delta)$-super-regular pair with $|X|=|Y|=N$, $x_1,x_2\in X\ (x_1\not =x_2)$, $y_1,y_2\in Y\ (y_1\not =y_2)$ and $l^i$ is an even integer with $4\leq l^i\leq 2N-4$ $(i=1,2)$, $l^1+l^2=2N$, then there are two vertex-disjoint paths $P_1$ and $P_2$ in $G$ such that the end vertices of $P_i$ are $x_i, y_i$ and $|V(P_i)|=l^i$ $(i=1,2)$.
\end{lemma}

\begin{proof}
Let $X^*=X-\{x_1,x_2\}$, $Y^*=Y-\{y_1,y_2\}$ and $H=H_1\cup H_2$ be the union of two vertex-disjoint paths $H_1, H_2$ satisfied $|V(H_i)|=l^i-2\ (i=1,2)$. It is not hard to see that $H$ can be embedded in $K_{N-2,N-2}$. By Slicing Lemma, we know that $(X^*,Y^*)$ is also a super-regular pair. Fix the end vertices of $H_1$ and $H_2$ to be the special vertices. For $H_i$, one of its end vertices is equipped with the neighbor set of $x_i$ and the other end vertex is equipped with the neighbor set of $y_i$ ($i=1,2$). By Lemma \ref{blowup}, $H$ can be embedded in $(X^*,Y^*)$ satisfied the restrictions of the special vertices. Since one of the end vertices of $H_i$ is a neighbor of $x_i$ and the other end vertex of $H_i$ is a neighbor of $y_i$, we can extend $H_i$ to a path $P_i$ with end vertices $x_i$ and $y_i$ ($i=1,2$). Then $P_1\cup P_2$ is a spanning subgraph of $G$ and $|V(P_i)|=l^i$ $(i=1,2)$.
\end{proof}

%

\section{Overview of the proof}
Recently many long-standing conjectures about Hamiltonian problems are proved or partially proved by using the Regularity Lemma (see \cite{osthus}, \cite{seymourconj} for some nice results). In our proof for Theorem \ref{main} we will use the Regularity Lemma-Blow-up Lemma method as many other studies (see \cite{chen}, \cite{sarkozy} for some similar ideas).

For proving Theorem \ref{main}, we say that we only need to consider the graphs of even order. Actually, we claim that if the conclusion of Theorem \ref{main} is true for graphs of even order, then it is also true for graphs of odd order. Precisely, for a graph $G$ of odd order $n$, we choose one vertex $v$ in $G$ which is not either of the two vertices $x,y$, then the minimum degree of graph $G^*=G-\{v\}$ is at least $\lceil \frac{n}{2}\rceil=\frac{n-1}{2}+1$. Since $G^*$ is a graph of even order, by assumption we can locate $x,y$ with distance $\frac{n-1}{2}$ on a Hamiltonian cycle $C^*$ of $G^*$. By the degree condition of $v$ in $G$, there exsit two consecutive vertices $u_1,u_2$ on $C^*$ which are adjacent to $v$. Replacing the edge $u_1u_2$ on $C^*$ by the 2-path $u_1vu_2$, we obtain a Hamiltonian cycle of $G$ in which $x,y$ have distance $\frac{n-1}{2}=\lfloor \frac{n}{2}\rfloor$.

Now let us consider a graph $G$ of even order $n$ with
\begin{equation}\label{deg}
\delta (G)\geq \frac{n}{2}+1.
\end{equation}

We assume that $n$ is sufficiently large and we fix the following sequence of parameters,
\begin{equation}\label{para}
0<\epsilon \ll d\ll \alpha \ll1.
\end{equation}
Here $a \ll b$ means $a$ is sufficiently small compared to $b$. For simplicity, we don't specify their dependencies in the proof, although we could.

A {\em balanced partition} of $V(G)$ into $V_1$ and $V_2$ is a partition of $V(G)=V_1\cup V_2$ such that $|V_1|=|V_2|=\frac{n}{2}$. We define two extremal cases as follows.

{\bf Extremal Case 1}: There exists a balanced partition of $V(G)$ into $V_1$ and $V_2$ such that the density $d(V_1,V_2)\geq 1-\alpha$.

{\bf Extremal Case 2}: There exists a balanced partition of $V(G)$ into $V_1$ and $V_2$ such that the density $d(V_1,V_2)\leq \alpha$.

The proof of Theorem \ref{main} will be divided into two parts: the non-extremal case part in Section 4 and the extremal cases part in Section 5. Indeed, due to the parity of $\frac{n}{2}$, our proof will have some cases discussions.

%
%
%
%
%
%

\section{Non-extremal case}

\subsection{Applying the Regularity Lemma}
Let $G$ be a graph not either of the extremal cases and the vertices $x$, $y$ have been chosen. We apply the Regularity Lemma in $G$ with parameter $\epsilon $ and $d$ as in (\ref{para}). We get a partition of $V(G)$ into $l+1$ clusters $V_0, V_1, V_2,...,V_l$. Assume that $l$ is even, if not, we move the vertices of one of the clusters into $V_0$ to make $l$ be an even number. Now $|V_0|\leq 2\epsilon n$ and $lL\geq (1-2\epsilon )n$. Let $k:=\frac{l}{2}$.

We define the following {\em reduced graph} $R$: the vertices of $R$ are $r_1, r_2,..., r_l$, and there is an edge between $r_i$ and $r_j$ if the pair $(V_i, V_j)$ is $\epsilon $-regular in $G^{'}$ with density exceeding $d$. If no ambiguity arises, we won't distinguish the cluster and its corresponding vertex in $R$.

The following claim shows that $R$ inherits the minimum degree condition.

\begin{claim}\label{mdor}
$\delta (R)\geq (\frac{1}{2}-2d)l$.
\end{claim}

\begin{proof}
For any cluster $V_i$ ($i\geq 1$), the neighbors of $v\in V_i$ in $G^{'}$ can only be in $V_0$ and in the clusters which are neighbors of $V_i$ in $R$. So for $V_i$,
\begin{equation*}
(\frac{n}{2}+1-(d+\epsilon )n)L\leq \sum\limits_{v\in V_i}deg_{G^{'}}(v)\leq 2\epsilon nL+deg_{R}(r_i)L^2.
\end{equation*}
Thus $deg_{R}(r_i)\geq (\frac{1}{2}-d-3\epsilon )\frac{n}{L}>(\frac{1}{2}-2d)l$ provided $3\epsilon <d$.
\end{proof}

By Claim 4.5 in \cite{chen}, we can get a similar claim as follows. It shows that there exists an upper bound of the independent number of $R$. The proof of this lemma is almost the same as the proof of Claim 4.5 in \cite{chen}, so we omit it here. We need to mention that there are some differences of the parameters between our paper and \cite{chen}, but it won't affect the conclusion.

\begin{claim}\label{independent}
$G$ is a graph which is not either of the extremal cases, then \\
(1) the independent number of $R$ is less than $(\frac{1}{2}-8d)l$,\\
(2) $R$ contains no two disjoint subsets $R_1$, $R_2$ of size at least $(\frac{1}{2}-6d)l$ such that $e_R(R_1,R_2)=0$.
\end{claim}

By Claim \ref{independent}, we can say $R$ is Hamiltonian. There are also some similar arguments in \cite{chen}. We just give the claim without proofs (see \cite{chen} for more details).

\begin{claim}
$R$ is a Hamiltonian graph.
\end{claim}

%
%

\subsection{Constructing paths to connect clusters}

We call a vertex $v$ {\em friendly} to a cluster $X$ if $deg_G(v,X)\geq (d-\epsilon)|X|$. Moreover, given an $\epsilon $-regular pair $(X,Y)$ of clusters and a subset $Y^{'}\subseteq Y$, we call a vertex $v\in X$ {\em friendly} to $Y^{'}$, if $deg(v,Y^{'})\geq (d-\epsilon )|Y^{'}|$. Actually by Lemma \ref{large}, at most $\epsilon |X|$ vertices of $X$ are not friendly to $Y^{'}$ whenever $|Y^{'}|>\epsilon |Y|$.

\begin{claim}\label{friendly}
Every vertex $v\in V(G)$ is friendly to at least $(\frac{1}{2}-2d)l$ clusters.
\end{claim}

\begin{proof}
Assume for a contradiction that there are less than $(\frac{1}{2}-2d)l$ friendly clusters for $v$. Then
\begin{align*}
deg_G(v) \leq (\frac{1}{2}-2d)lL+(d-\epsilon )Ll+2\epsilon n\leq (\frac{1}{2}-d+2\epsilon )n<\frac{n}{2}
\end{align*}
provided that $2\epsilon <d$, which is a contradiction to (\ref{deg}).
\end{proof}

Let $C_R$ be a Hamiltonian cycle in $R$. We choose two distinct clusters $X$, $Y$ which are as close as possible on $C_R$ such that $x$ is friendly to $Y$ and $y$ is friendly to $X$.

\begin{claim}\label{close}
We can choose distinct clusters $X$ and $Y$ such that $x$ is friendly to $Y$, $y$ is friendly to $X$ and $dist_{C_R}(X,Y)\leq 3dl$.
\end{claim}

\begin{proof}
Let $\mathcal{X}$ be the the family of friendly clusters for $x$ and $\mathcal{Y}$ be the the family of friendly clusters for $y$. By Claim \ref{friendly}, $|\mathcal{X}|,|\mathcal{Y}|\geq (\frac{1}{2}-2d)l$. We won't distinguish a cluster and its corresponding vertex on $C_R$.

We call a segment on $C_R$ a {\em $\mathcal{X}$-segment} if it is a maximal segment(or we can say a maximal path) on $C_R$ with both end vertices in $\mathcal{X}$ such that it contains no clusters in $\mathcal{Y}$. Similarly, a {\em $\mathcal{Y}$-segment} is a maximal segment on $C_R$ with both end vertices in $\mathcal{Y}$ such that it contains no clusters in $\mathcal{X}$. Each cluster in $\mathcal{X}\cap\mathcal{Y}$ forms a $\mathcal{X}$-segment($\mathcal{Y}$-segment) with one vertex on $C_R$. Now $C_R$ is divided by all these segments. We choose $X\in \mathcal{X}$ and $Y\in \mathcal{Y}$ such that $X$ and $Y$ are two closest end vertices in two continuous segments on $C_R$.

Suppose $\mathcal{X}\cap\mathcal{Y}$ is equal to $\mathcal{X}$ or $\mathcal{Y}$. Thus $|\mathcal{X}\cap\mathcal{Y}|\geq (\frac{1}{2}-2d)l$. The distance between $X$ and $Y$ should have
$$dist_{C_R}(X,Y)\leq \frac{l-|\mathcal{X}\cap\mathcal{Y}|}{|\mathcal{X}\cap\mathcal{Y}|}+1\leq \frac{l-(\frac{1}{2}-2d)l}{(\frac{1}{2}-2d)l}+1\leq \frac{8d}{1-4d}+2\leq 4$$
provided $d\leq \frac{1}{8}$. Since the distance between two vertices in a path is the number of internal vertices plus one, we have a ``+1'' in the above calculation.

Suppose $\mathcal{X}\cap\mathcal{Y}$ is not equal to either of $\mathcal{X}$ and $\mathcal{Y}$. The number of segments should be no less than $|\mathcal{X}\cap\mathcal{Y}|+2$.

$$dist_{C_R}(X,Y)\leq \frac{l-(|\mathcal{X}|+|\mathcal{Y}|-|\mathcal{X}\cap\mathcal{Y}|)}{|\mathcal{X}\cap\mathcal{Y}|+2}+1\leq \frac{4dl+|\mathcal{X}\cap\mathcal{Y}|}{|\mathcal{X}\cap\mathcal{Y}|+2}+1\leq 2dl+2.$$

Since $n=Ll+|V_0|\leq (l+2)\epsilon n$, we have $l\geq \frac{1}{\epsilon }-2\geq \frac{2}{d}$, provided $\epsilon \leq \frac{d}{4}$. Then $dl\geq 2$, $dist_{C_R}(X,Y)\leq 3dl$.
\end{proof}

By Claim \ref{close} we choose these two clusters $X$ and $Y$. We give a new notation for all clusters except $V_0$. We choose a direction of $C_R$, which is along the longer path from $Y$ to $X$ on $C_R$ (there are two paths from $Y$ to $X$ on $C_R$, and we choose the longer one), then starting from $Y$, we denote the clusters by $Y_1$, $X_2$, $Y_2$, $X_3$, $Y_3$,..., $X_k$, $Y_k$, $X_1$ along this direction (recall that $k=\frac{l}{2}$). $Y$ is denoted by $Y_1$ and $X$ is denoted by a $X_i$ or a $Y_i$.

We need to mention that the parity of $\frac{n}{2}$ and the new notation of $X$ would affect our following discussions. In the following arguments, we assume that $\frac{n}{2}$ is even and $X$ is denoted by $Y_t$. We call this the non-extremal case 1. For the other cases ($\frac{n}{2}$ is odd or $X$ is denoted by some $X_i$), we will discuss them in subsection 4.5.

We know that $t\not =1$. By Claim \ref{close}, $dist_{C_R}(Y_1,Y_t)=l-2t+2\leq 3dl$. So
\begin{equation}\label{t}
t-1\geq \frac{1-3d}{2}l.
\end{equation}
This will be used in subsection 4.4. We call $X_i$, $Y_i$ {\em partners} of each other ($1\leq i\leq k$).

Now we construct some paths to connect $Y_i$ and $X_{i+1}$ ($1\leq i\leq k$). We always say $X_{k+1}=X_1$.

Since $x$ is friendly to $Y_1$, we can choose two neighbors of $x$ in $Y_1$, denoted by $w_x$ and $y_1^1$, such that $w_x$ is friendly to $X_2$ and $y_1^1$ is friendly to $X_1$. Indeed, $x$ has at least $(d-\epsilon)L$ neighbors in $Y_1$ and $(X_1,Y_1)$,$(Y_1,X_2)$ are both regular pairs, so at least $(d-\epsilon)L-\epsilon L$ vertices of $Y_1$ can be chosen as $w_x$ and $y_1^1$. Choose a neighbor of $w_x$ in $X_2$, denoted by $x_2^1$, such that $x_2^1$ is friendly to $Y_2$. We know that at least $(d-\epsilon)L-\epsilon L$ vertices of $X_2$ can be chosen as $x_2^1$. And we have a path from $y_1^1$ to $x_2^1$, precisely $P_1:=y_1^1xw_xx_2^1$. We call this procedure {\em joining $x$} to $Y_1$. Similarly we can construct a path $P_t=y_t^1yw_yx_{t+1}^1$, where $y_t^1\in Y_t$ is friendly to $X_t$, $x_{t+1}^1\in X_{t+1}$ is friendly to $Y_{t+1}$ and $w_y\in Y_t$. We call this procedure {\em joining $y$} to $Y_t$. For $1\leq i\leq k$ and $i\not =1,t$, we choose two adjacent vertices $y_i^1$ and $x_{i+1}^1$ such that $y_i^1\in Y_i$ is friendly to $X_i$ and $x_{i+1}^1\in X_{i+1}$ is friendly to $Y_{i+1}$ (we always use $x_1^1$ to denote $x_{k+1}^1$). It is possible to find these vertices as the argument for $P_1$. Let $P_i$ ($1\leq i\leq k$ and $i\not =1,t$) be the path $y_i^1x_{i+1}^1$, which connects $Y_i$ and $X_{i+1}$. We always call the vertices in $P_i$ ($1\leq i\leq k$) {\em used} vertices.

We need some other vertex-disjoint paths to connect $Y_i$ and $X_{i+1}$ ($1\leq i\leq k$). By the same method, we choose two adjacent unused vertices $y_i^2$ and $x_{i+1}^2$ such that $y_i^2\in Y_i$ is friendly to $X_i$ and $x_{i+1}^2\in X_{i+1}$ is friendly to $Y_{i+1}$ (we always use $x_1^2$ to denote $x_{k+1}^2$). Let $Q_i$ ($1\leq i\leq k$) be the path $y_i^2x_{i+1}^2$, which connects $Y_i$ and $X_{i+1}$. By Lemma \ref{large}, it is possible to find these unused vertices.

Summary that, we have constructed paths $P_i$ and $Q_i$ ($1\leq i\leq k$), which are vertex-disjoint and connect $Y_i$ and $X_{i+1}$ (see Figure 1). $x$ is on $P_1$ and $y$ is on $P_t$. Every end vertex of these paths is friendly to its cluster's partner. We use $INT$ to denote the vertex set of all internal vertices on all $P_i$'s and $Q_i$'s. Now $INT=\{x,y,w_x,w_y\}$.

\begin{figure}[htbp]
\centering
\includegraphics[width=10cm]{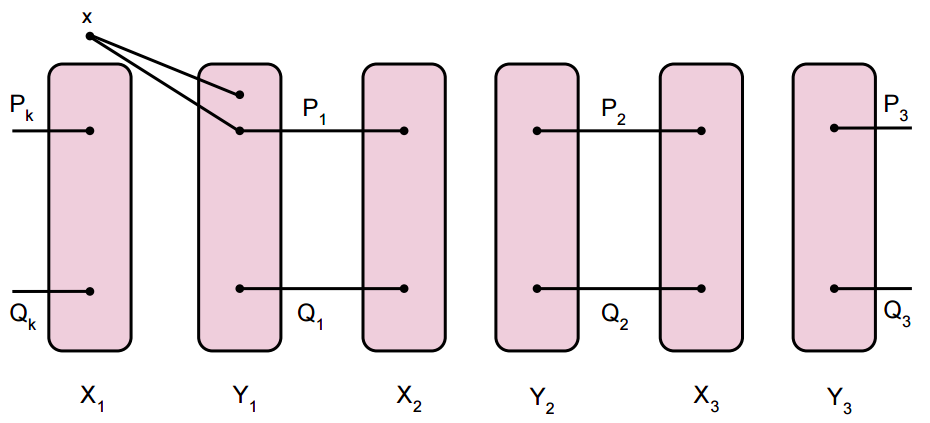}
\caption{Construction of $P_i$'s and $Q_i$'s.}
\end{figure}

For every $1\leq i\leq k$, let

$X_i^{'}:=\{u\in X_i: deg(u,Y_i)\geq (d-\epsilon )L\}$, $Y_i^{'}:=\{v\in Y_i: deg(v,X_i)\geq (d-\epsilon )L\}$.

Since $(X_i,Y_i)$ is $\epsilon $-regular, we have $|X_i^{'}|, |Y_i^{'}|\geq (1-\epsilon )L$. We move all the vertices in $X_i-X_i^{'}$ and $Y_i-Y_i^{'}$ to $V_0$. Meanwhile, we need to make sure $(X_i^{'},Y_i^{'})$ is balanced, so by Lemma \ref{large} we may move at most $\epsilon Ll$ vertices in $(X_i^{'},Y_i^{'})$ to guarantee that. We also remove all the vertices in $INT$ out of $V_0$, $X_i^{'}$ and $Y_i^{'}$. This may cause that some $(X_i^{'},Y_i^{'})$ is not balanced. For example, if $w_x\in Y_1^{'}$, we remove it from $Y_1^{'}$ and make $(X_1^{'},Y_1^{'})$ be not balanced. In this example, to make sure $(X_1^{'},Y_1^{'})$ is balanced, we move a vertex in $X_i^{'}$ to $V_0$. We do the same operations for all the vertices in $INT$. Since $|INT|=4$ and $n$ is sufficiently large, at most $\epsilon Ll+4\leq 2\epsilon n$ vertices are moved to $V_0$. We derive that $|V_0|\leq 4\epsilon n$ and $|X_i^{'}|=|Y_i^{'}|\geq (1-\epsilon )L-1$ ($1\leq i\leq k$) in this step. Since $\epsilon L\geq \epsilon \frac{(1-2\epsilon)n}{l}\geq \frac{\epsilon (1-2\epsilon )}{M_0}n$ and $n$ is sufficiently large, we say $\epsilon L\geq 1$. Thus $|X_i^{'}|=|Y_i^{'}|\geq (1-2\epsilon )L$ ($1\leq i\leq k$). The minimum degree in each pair is at least $(d-\epsilon )L-\epsilon L-1\geq (d-3\epsilon )L$.

\subsection{Handling of all the vertices of $V_0$}
In this step, we extend those paths $Q_i$'s ($1\leq i\leq k$) by adding all the vertices of $V_0$ to them.

If a vertex $v$ is friendly to a cluster $X$, we denote this relation by $v\sim X$. If two clusters $X$ and $Y$ are a regular pair, we denote this relation by $X\sim Y$. Given two vertices $u,v\in V(G)$, a $u,v$-chain of length $2s$ with distinct clusters $A_1,B_1,...,A_s,B_s$ is $u\sim A_1\sim B_1\sim \cdot \cdot \cdot \sim A_s\sim B_s\sim v$ and $\{A_j,B_j\}=\{X_i,Y_i\}$ for some $1\leq i\leq k$. We can say that for each pair of vertices in $V_0$ we have the following claim. There are also some similar discussions in \cite{chen}. Since we need a different bound to finish our proof, we give the claim as follows.

\begin{claim}
For each pair of vertices $\{u,v\}$ in $V_0$, we can find $u,v$-chains of length at most four such that every cluster is used in at most $\frac{d}{10}L$ chains.
\end{claim}

\begin{proof}
We deal with the vertices of $V_0$ pair by pair. Suppose we have found the desired chains for $s$ pairs such that no cluster is used in more than $\frac{d}{10}L$ chains. Since $|V_0|\leq 4\epsilon n$, $s<2\epsilon n$. Let $\mathcal{O}$ be the set of clusters which are used $\frac{d}{10}L$ times.

We have a bound on the cardinality of $\mathcal{O}$,
$$\frac{d}{10}L|\mathcal{O}|\leq 4s\leq 8\epsilon n\leq 8\epsilon \frac{2kL}{1-2\epsilon }.$$

So $|\mathcal{O}|\leq \frac{160\epsilon k}{(1-2\epsilon )d}\leq \frac{160\epsilon l}{d}\leq dl$, provided $d^2\geq 160\epsilon$.

Now consider an unused pair $\{u,v\}$ in $V_0$, we try to find a $u,v$-chain of length at most four such that every cluster is used in at most $\frac{d}{10}L$ chains. Let $\mathcal {U}$ be the set of clusters which are friendly to $u$ and not in $\mathcal{O}$. Similarly, let $\mathcal {V}$ be the set of clusters which are friendly to $v$ and not in $\mathcal{O}$. Let $P(\mathcal{U})$ and $P(\mathcal{V})$ be the set of partners of clusters in $\mathcal {U}$ and $\mathcal {V}$ respectively. It is easy to see that $|\mathcal {U}|=|P(\mathcal{U})|$ and $|\mathcal {V}|=|P(\mathcal{V})|$. Moreover, since $|\mathcal{O}|\leq dl$, by Claim \ref{friendly}, we know that $|\mathcal {U}|=|P(\mathcal{U})|\geq (\frac{1}{2}-3d)l$ and $|\mathcal {V}|=|P(\mathcal{V})|\geq (\frac{1}{2}-3d)l$.

If $e_R(P(\mathcal{U}),P(\mathcal{V}))\not =0$, it is not hard to see that we can find a $u,v$-chain of length two or four.

Now we assume $e_R(P(\mathcal{U}),P(\mathcal{V}))=0$. If $P(\mathcal{U})\cap P(\mathcal{V})=\emptyset $, then we find two disjoint vertex subsets of $R$ with size more than $(\frac{1}{2}-3d)l$, which is a contradiction to Claim \ref{independent}.

We assume that there exists a cluster $X\in P(\mathcal{U})\cap P(\mathcal{V})$. So $deg_R(X)\geq (\frac{1}{2}-2d)l$. Since $e_R(P(\mathcal{U}),P(\mathcal{V}))=0$, $X$ is not adjacent to any cluster in $P(\mathcal{U})\cup P(\mathcal{V})$. Thus $|P(\mathcal{U})\cup P(\mathcal{V})|\leq (\frac{1}{2}+2d)l$. Since $|P(\mathcal{U})|\geq (\frac{1}{2}-3d)l$ and $|P(\mathcal{V})|\geq (\frac{1}{2}-3d)l$, $|P(\mathcal{U})\cap P(\mathcal{V})|\geq (\frac{1}{2}-8d)l$. $P(\mathcal{U})\cap P(\mathcal{V})$ is an independent set in $R$, which is a contradiction to Claim \ref{independent}.
\end{proof}

Now we extend those paths $Q_i$'s by using vertices of $V_0$. Recall that the end vertices of $Q_i$ are $y_i^2\in Y_i$ and $x_{i+1}^2\in X_{i+1}$. We deal with the vertices of $V_0$ pair by pair. Assume that we deal with the pair $(u,v)$ now.

If there is a chain of length two between $u$ and $v$, we assume that this chain is $u\sim X_i\sim Y_i\sim v$, for some $1\leq i\leq k$. We choose two adjacent vertices $w_1\in X_i^{'}$ and $w_2\in Y_i^{'}$ such that $w_1$ is a neighbor of $y_i^2$ and $w_2$ is a neighbor of $v$. Since $y_i^2$ is friendly to $X_i$ and $v$ is friendly to $Y_i$, the size of the neighbor sets of $w_1$ and $w_2$ are at least $(d-3\epsilon )L$. By Lemma \ref{large}, it is possible to choose $w_1$ and $w_2$. We choose another neighbor of $v$ in $Y_i^{'}$, denoted by $w_3$. Then we extend $Q_i$ to $Q_i\cup \{w_3v,vw_2,w_2w_1,w_1y_i^2\}$. We still denote this new path by $Q_i$ and call $w_3$ the new $y_i^2$ to make sure that the end vertices of the new $Q_i$ are denoted by $y_i^2\in Y_i$ and $x_{i+1}^2\in X_{i+1}$. Similarly, for $u$, we can choose $w_5, w_6\in X_i^{'}$ and $w_4\in Y_i^{'}$ to extend $Q_{i-1}$ to $Q_{i-1}\cup \{x_i^2w_4,w_4w_5,w_5u,uw_6\}$. We still denote this new path by $Q_{i-1}$ and call $w_6$ the new $x_i^2$ to make sure that the end vertices of $P_{i-1}$ are  $y_{i-1}^2\in Y_{i-1}$ and $x_i^2\in X_{i}$. And we update the set $INT$. Indeed, three vertices of $X_i^{'}$ are added to $INT$ (also for $Y_i^{'}$) and totally eight vertices are added to $INT$ including $u,v$. Since $v$ behaves like a vertex in $X_i^{'}$ and $u$ behaves like a vertex in $Y_i^{'}$, we call this procedure {\em inserting $v$} into $X_i^{'}$ to extend $Q_i$ and {\em inserting} $u$ into $Y_i^{'}$ to extend $Q_{i-1}$ (see Figure 2).

\begin{figure}[htbp]
\centering
\includegraphics[width=9cm]{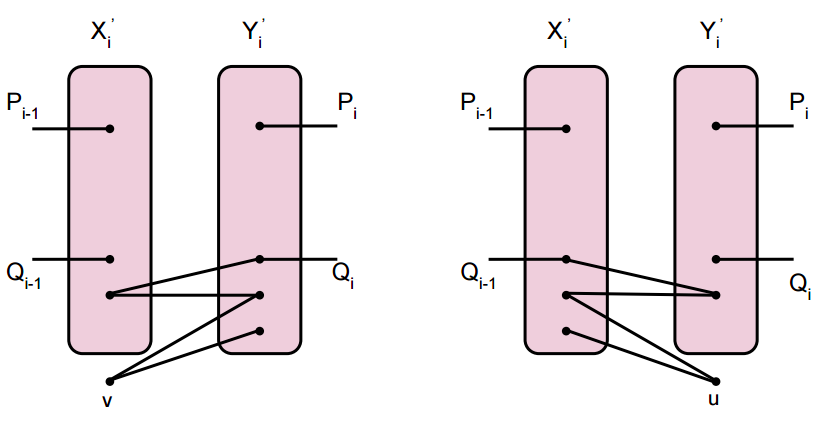}
\caption{Extending $Q_{i-1}$ and $Q_i$ when $u,v$ have a chain of length two.}
\end{figure}

Now we consider that the $u,v$-chain has length four. Without loss of generality, we assume that the chain is $u\sim X_i\sim Y_i\sim X_j\sim Y_j\sim v$, for some $i,j$. We extend the path $Q_{i-1}$ by inserting $u$ into $Y_i^{'}$. We choose a vertex of $Y_i^{'}$ which is friendly to $X_j$ and insert it into $Y_j^{'}$ to extend $Q_{j-1}$. At last we extend the path $Q_{j}$ by inserting $v$ into $X_{j}^{'}$. Meanwhile, we update the set $INT$. Indeed, two vertices of $X_i^{'}$ are added to $INT$ (also for $Y_i^{'}$) and three vertices of $X_j^{'}$ are added to $INT$ (also for $Y_j^{'}$). So totally twelve vertices are added into $INT$ including $u,v$.

We continue this process till there is no vertices left in $V_0$. Denote $X_i^{*}=X_i^{'}-INT$ and $Y_i^{*}=Y_i^{'}-INT$. It is not hard to see that the pair ($X_i^{*},Y_i^{*}$) is still balanced. For inserting each pair of vertices, at most three vertices of a cluster in the chain are used. So $$|X_i^{*}|=|Y_i^{*}|\geq (1-2\epsilon )L-3\frac{d}{10}L\geq (1-\frac{d}{2})L$$
provided $\epsilon <\frac{d}{10}$.

For each vertex $u\in X_i^{*}$, we have
$$deg(u,Y_i^{*})\geq (d-3\epsilon )L-3\frac{d}{10}L\geq \frac{d}{2}L$$
provided $\epsilon <\frac{d}{15}$. And it is the same for the degree of any vertex in $Y_i^{*}$.

Thus by Slicing Lemma, we can say the pair ($X_i^{*},Y_i^{*}$) is ($2\epsilon ,\frac{d}{2}$)-super-regular ($1\leq i\leq k$).

\subsection{Constructing the desired Hamiltonian cycle}

In this step, first we use Lemma \ref{blow2} to construct two paths $W_i^1$ and $W_i^2$ in each pair ($X_i^*,Y_i^*$). Then we combine all these paths with $P_i$'s and $Q_i$'s to obtain a Hamiltonian cycle in $G$. At last we fix the length of $W_i^1$ and $W_i^2$ in each pair to make sure that $x$ and $y$ have distance $\frac{n}{2}$ on this Hamiltonian cycle.

For each $1\leq i\leq k$, we choose any even integers $l_i^1,l_i^2$ such that $4\leq l_i^1,l_i^2\leq 2|X_i^*|-4$ and $l_i^1+l_i^2=2|X_i^*|$. We will fix these integers later.

For $2\leq i\leq k$, by Lemma \ref{blow2}, we construct two paths $W_i^1$ and $W_i^2$ in the pair ($X_i^*,Y_i^*$) such that

\noindent
(a) $W_i^1$ has end vertices $x_i^1$ and $y_i^1$ with $|V(W_i^1)|=l_i^1$;

\noindent
(b) $W_i^2$ has end vertices $x_i^2$ and $y_i^2$ with $|V(W_i^2)|=l_i^2$.

And for $i=1$, we construct two paths $W_1^1$ and $W_1^2$ in the pair ($X_1^*,Y_1^*$) such that

\noindent
(c) $W_1^1$ has end vertices $x_1^1$ and $y_2^1$ with $|V(W_1^1)|=l_1^1$;

\noindent
(d) $W_1^2$ has end vertices $x_1^2$ and $y_1^2$ with $|V(W_2^1)|=l_2^1$.

It is not hard to see
$$C=P_1\cup (\bigcup_{i=2}^k (W_i^1\cup P_i))\cup W_1^1\cup Q_1\cup (\bigcup_{i=2}^k (W_i^2\cup Q_i))\cup W_1^2$$
is a Hamiltonian cycle in $G$.

\begin{figure}[htbp]
\centering
\includegraphics[width=10cm]{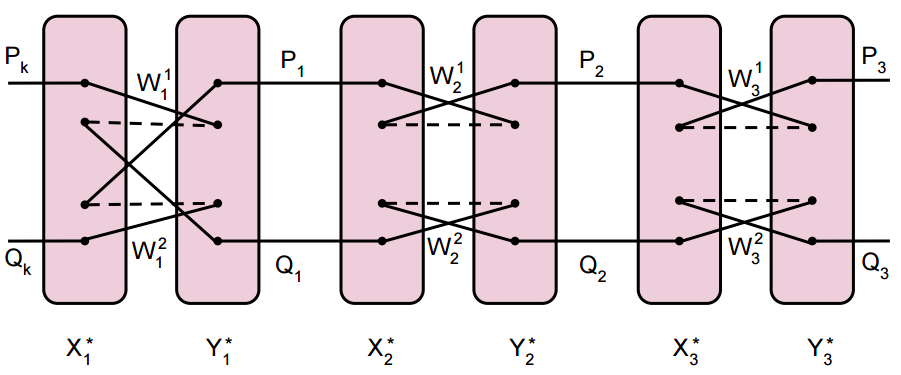}
\caption{Construction of the Hamiltonian cycle $C$.}
\end{figure}

To finish our proof, we need to make sure that $x$ and $y$ have distance $\frac{n}{2}$ on $C$. Our Hamiltonian cycle is constructed in a bipartite graph $(\bigcup X_i)\cup (\bigcup Y_i)$ (take $\bigcup X_i$ as a part and $\bigcup Y_i$ the other one). Since $x$ behaves like a vertex in $X_1$ and $y$ behaves like a vertex in $X_t$, the distance of $x$ and $y$ on $C$ should be an even number. Recall that we assume that $\frac{n}{2}$ is even. So there is no parity problem in this non-extremal case.

\begin{claim}\label{proper}
We can properly choose the value of $l_i^1$ ($2\leq i\leq t$) such that $dist_C(x,y)=\frac{n}{2}$.
\end{claim}

\begin{proof}
We consider this path $P:=P_1\cup (\bigcup_{i=2}^t (W_i^1\cup P_i))$ of the Hamiltonian cycle. We need the distance of $x$ and $y$ on $C$ to be $\frac{n}{2}$, so the vertex number between $x$ and $y$ on $P$ should be $\frac{n}{2}-1$. For the vertices between $x$ and $y$ on $P$, the only vertex not belong to $W_i^1$ ($2\leq i\leq t$) is $w_x$. Thus we need to make sure
\begin{equation}\label{sumofli}
\frac{n}{2}-1=\sum\limits_{i=2}^{t}l_i^1+1.
\end{equation}

Since by Lemma \ref{blow2}, $l_i^1$ can be any even integer such that $4\leq l_i^1\leq 2|X_i^*|-4$. By $|X_i^*|\geq (1-\frac{d}{2})L$, we can choose $l_i^1$ such that $\sum\limits_{i=2}^{t}l_i^1$ can be any even integer with the following bound,
$$4(t-1)\leq \sum\limits_{i=2}^{t}l_i^1\leq 2(t-1)(1-\frac{d}{2})L-4(t-1)=2(t-1)((1-\frac{d}{2})L-2).$$

We know that $t\leq k=\frac{l}{2}$, then $4(t-1)<2l$. Since $l\leq M_0$ in the Regularity Lemma and $n$ is sufficiently large (let $n\geq 4M_0+4$), we can say that $4(t-1)<2l\leq 2M_0\leq \frac{n}{2}-2$.

By (\ref{t}), we also know $t-1\geq \frac{1-3d}{2}l$, so
\begin{eqnarray*}
2(t-1)((1-\frac{d}{2})L-2)&\geq &(1-3d)(1-\frac{d}{2})lL-2(1-3d)l\\ &\geq &(1-\frac{7}{2}d)(1-2\epsilon )n-2l\\ &\geq &(1-4d)n-2M_0\geq \frac{3}{4}n-2M_0,
\end{eqnarray*}
provided $4\epsilon \leq d\leq \frac{1}{16}$ and $l\leq M_0$. Since $n$ is sufficiently large (let $n\geq 8M_0$), we can say that $2(t-1)((1-\frac{d}{2})L-2)\geq \frac{3}{4}n-2M_0\geq \frac{n}{2}$.

So we can choose the values of $l_i^1$ ($2\leq i\leq t$) such that $\sum\limits_{i=2}^{t}l_i^1=\frac{n}{2}-2$ satisfying (\ref{sumofli}).

We choose $l_i^1$ ($2\leq i\leq t$) such that (\ref{sumofli}) holds and arbitrarily choose the even integers  $l_1^1$, $l_i^1$ ($t<i\leq k$) with the conditions in Lemma \ref{blow2}. In this choice of $l_i^1,l_i^2$ ($1\leq i\leq k$), we can make sure that $dist_C(x,y)=\frac{n}{2}$.
\end{proof}

This leads to the end of the proof of the non-extremal case 1.

\subsection{Other non-extremal cases}

We discuss the other non-extremal cases. Suppose $\frac{n}{2}$ is even and $X$ (the cluster friendly to $y$) is denoted by $X_t$ ($1\leq t\leq k$) in the second step of the proof above. We call this the non-extremal case 2. It seems that the method above doesn't work. Since our Hamiltonian cycle is constructed in a bipartite graph $(\bigcup X_i)\cup (\bigcup Y_i)$ and $x$ (resp. $y$) behaves like a vertex in $X_1$ (resp. $Y_t$), we cannot locate $x$ and $y$ with the even distance $\frac{n}{2}$ on the Hamiltonian cycle.

We need some tricks to change the parity. Recall that, in the second step of the proof above, we construct $P_2=y_2^1x_3^1$ and $Q_2=y_2^2x_3^2$. Since $\delta (G)\geq \frac{n}{2}+1$, any two vertices have at least two common neighbors. Suppose that $y_2^1,x_3^1$ have a common neighbor $u_1$ and $y_2^2,x_3^2$ have a common neighbor $u_2\not =u_1$. First, if $u_1$ and $u_2$ are both different with $x$ and $y$, then in the second step of the proof we choose $P_2=y_2^1u_1x_3^1$ and $Q_2=y_2^2u_2x_3^2$. By the same method above we construct the Hamiltonian cycle. It is not hard to see that we can make sure the distance between $x$ and $y$ on this cycle is the even number $\frac{n}{2}$. Second, assume that we cannot find these required $y_2^1, x_3^1, y_2^2, x_3^2$ such that $u_1$ and $u_2$ are both different with $x$ and $y$. We choose $y_2^1$ to be friendly to $X_2$ and $X_3$. There are at least $(1-2\epsilon )L$ possible choices for $y_2^1$ and it is similar for $y_2^2$. We choose $x_3^1$ to be a neighbor of $y_2^1$ and friendly to $Y_2$ and $Y_3$. There are at least $(d-\epsilon )L-2\epsilon L=(d-3\epsilon )L$ possible choices for $x_3^1$ and it is similar for $x_3^2$. By the assumption every these possible $y_2^1, x_3^1, y_2^2, x_3^2$ should be neighbors of $x$ or $y$. So $x$ or $y$ should have at least $\frac{1}{2}(d-3\epsilon )L$ neighbors in $Y_2$ and also in $X_3$. If $y$ has at least $\frac{1}{2}(d-3\epsilon )L$ neighbors in $Y_2$ and in $X_3$, we change the choice of $X$ and choose $Y_2$ to be the $X$. By the degree of $y$ in $Y_2$ we can join $y$ to $X=Y_2$ and we also join $x$ in $Y=Y_1$ as before. Now $x$ and $y$ both behave like vertices in $(\bigcup X_i)$. Otherwise, $x$ has at least $\frac{1}{2}(d-3\epsilon )L$ neighbors in $Y_2$ and in $X_3$. We change the choice of $Y$ and choose $X_3$ to be the $Y$. Since we can join $x$ to $Y=X_3$ and join $y$ to $X=X_t$, $x$ and $y$ both behave like vertices in $(\bigcup Y_i)$. By the choice of $X$ and $Y$, we give a new notation for all clusters and continue the proof as before. Although the calculation in Claim \ref{close} and Claim \ref{proper} will have some minor differences, it won't affect the conclusion.

We consider the cases when $\frac{n}{2}$ is odd. Actually in the second step of the proof, if the selected clusters $X$ and $Y$ belong to the different parts of the bipartite graph $(\bigcup X_i)\cup (\bigcup Y_i)$, the proof is similar to non-extremal case 1. If the selected clusters $X$ and $Y$ belong to the same part of the bipartite graph $(\bigcup X_i)\cup (\bigcup Y_i)$, the proof is similar to non-extremal case 2. We omit these similar proofs here.

\section{Extremal cases}

\subsection{Extremal case 1}
Suppose $G$ is a graph on $n$ vertices with $\delta (G)\geq \frac{n}{2}+1$ and there exists a balanced partition of $V(G)$ into $V_1$ and $V_2$ such that the density $d(V_1,V_2)\geq 1-\alpha$.
We suppose $\alpha \leq (\frac{1}{9})^3$. Let $\alpha_1=\alpha ^{\frac{1}{3}}$ and $\alpha_2=\alpha ^{\frac{2}{3}}$. So $\alpha_1 \geq 9\alpha_2 $.

We need the following lemma to continue our proof.

\begin{lemma}\label{ec1}
If $G$ is in extremal case 1, then $G$ contains a balanced spanning bipartite subgraph $G^*$ with parts $U_1$, $U_2$ and $G^{*}$ has the following properties:

(a) there is a vertex set $W$ such that there exist vertex-disjoint 2-paths (paths of length two) in $G^{*}$ with the vertices of $W$ as the middle vertices (not the end vertices) in each 2-path and $|W|\leq \alpha_2n$;

(b) $deg_{G^*}(v)\geq (1-\alpha_1-2\alpha_2)\frac{n}{2}$ for all $v\not \in W$.
\end{lemma}

\begin{proof}
For $i=1,2$, let $V_i^*=\{v\in V_i: deg(v,V_{3-i})\geq (1-\alpha_1)\frac{n}{2}\}$.

We claim that $|V_i-V_i^*|\leq \alpha_2\frac{n}{2}$. Otherwise
$$d(V_1,V_2)< \frac{\frac{\alpha_2n}{2}(1-\alpha_1)\frac{n}{2}+\frac{n}{2}(\frac{1}{2}-\frac{\alpha_2}{2})n}{(\frac{n}{2})^2}=\alpha_2(1-\alpha_1)+(1-\alpha_2)=1-\alpha,$$
which is a contradiction. So $|V_i^*|\geq (1-\alpha_2)\frac{n}{2}$.

For any vertex $v\in V_i-V_i^*$, if $deg(v,V_i)\geq (1-\alpha_1)\frac{n}{2}$, we also add it to $V_{3-i}^*$. We denote the two resulting sets by $V_i^{'}$ ($i=1,2$) and let $V_0=V-V_1^{'}-V_2^{'}$. We have $|V_0|\leq \alpha_2n$. For every vertex $v$ in $V_i^{'}$,
\begin{equation}\label{lemmadegree}
deg(v,V_{3-i}^{'})\geq (1-\alpha_1)\frac{n}{2}-\alpha_2\frac{n}{2}.
\end{equation}

For every vertex $u$ in $V_0$,
$$deg(u,V_i^{'})> (\frac{n}{2}-(1-\alpha_1)\frac{n}{2})-\alpha_2\frac{n}{2}\geq (\alpha_1-\alpha_2)\frac{n}{2}.$$

First, we assume $|V_1^{'}|,|V_2^{'}|\leq \frac{n}{2}$. Let $W=V_0$ and we add all the vertices in $V_0$ to $V_1^{'}$ and $V_2^{'}$ such that the final two sets are of the same size. We denote the final two sets by $U_1$ and $U_2$ corresponding to $V_1^{'}$ and $V_2^{'}$ respectively. Let $W_1=U_1-V_1^{'}$ and $W_2=U_2-V_2^{'}$. Thus $W=W_1\cup W_2$. Since for each vertex $u\in W_i$, $deg(u,V_{3-i}^{'})>(\alpha_1-\alpha_2)\frac{n}{2}\geq 2\alpha_2n\geq 2|W_i|$, we can greedily choose two neighbors of $u$ in $V_{3-i}^{'}$ such that the neighbors of all the vertices of $W_i$ are distinct ($i=1,2$). So $W,U_1,U_2$ are what we need. The degree condition is $deg_{G^*}(v)\geq (1-\alpha_1-\alpha_2)\frac{n}{2}$ by (\ref{lemmadegree}) for all $v\not \in W$.

Second, without loss of generality we assume $|V_1^{'}|>\frac{n}{2}$. Let $W_1$ be the set of vertices $v\in V_1^{'}$ such that $deg(v,V_1^{'})\geq \alpha_1\frac{n}{2}$.

If $|W_1|\geq |V_1^{'}|-\frac{n}{2}$, we take $W$ to be the set of all vertices of $V_0$ and arbitrary $|V_1^{'}|-\frac{n}{2}$ vertices of $W_1$. Let $U_1=V_1^{'}-W$ and $U_2=V_2^{'}\cup W$. We know that $|W|\leq \alpha_2\frac{n}{2}$. So for every vertex $u\in W$, we have
\begin{equation*}
deg(u,U_1)>(\alpha_1-\alpha_2)\frac{n}{2}-\alpha_2\frac{n}{2}\geq \alpha_2n\geq 2|W|.
\end{equation*}
Similarly, we can greedily choose two neighbors of $u$ in $U_1$ such that the neighbors of all the vertices of $W$ are distinct. The degree condition is $deg_{G^*}(v)\geq (1-\alpha_1-\alpha_2)\frac{n}{2}-\alpha_2\frac{n}{2}=(1-\alpha_1-2\alpha_2)\frac{n}{2}$ by (\ref{lemmadegree}) for all $v\not \in W$.

Now we assume $|W_1|< |V_1^{'}|-\frac{n}{2}$. Let $U_1=V_1^{'}-W_1$ and $U_2=V_2^{'}\cup V_0\cup W_1$. Let $t=|U_1|-\frac{n}{2}$, so $t\leq \alpha_2\frac{n}{2}$. Considering the induced graph $G[U_1]$, we know that
$$\delta (G[U_1])\geq \delta (G)-|U_2|\geq \frac{n}{2}+1-(\frac{n}{2}-t)\geq t+1;$$
$$\Delta(G[U_1])\leq \alpha_1\frac{n}{2}.$$

Suppose $G[U_1]$ has a biggest family of vertex-disjoint 2-paths on a vertex set $S$ and the number of those vertex-disjoint 2-paths is $s$. We consider the number of edges between $S$ and $G[U_1]-S$. So
$$t(\frac{n}{2}-3s)\leq \delta (G[U_1])(|U_1|-3s)\leq 3s\Delta(G[U_1])\leq 3s\alpha_1\frac{n}{2}.$$
We can get
\begin{equation}\label{s}
s\geq \frac{\frac{n}{2}t}{3(t+\alpha_1\frac{n}{2})}\geq \frac{\frac{n}{2}t}{3(\alpha_2\frac{n}{2}+\alpha_1\frac{n}{2})}\geq \frac{t}{3(\alpha_2+\alpha_1)}>t.
\end{equation}

So $G[U_1]$ has at least $t$ vertex-disjoint 2-paths. We choose $t$ vertex-disjoint 2-paths in $G[U_1]$ and move the middle vertices of all these vertex-disjoint 2-paths to $U_2$. Now $|U_1|=|U_2|=\frac{n}{2}$. Let $W$ be the union of $V_0\cup W_1$ and all these middle vertices. For any vertex $u\in V_0\cup W_1$,
$$deg(u,U_1)-3\alpha_2\frac{n}{2}>(\alpha_1-\alpha_2)\frac{n}{2}-\alpha_2\frac{n}{2}-3\alpha_2\frac{n}{2}\geq 2|V_0\cup W_1|.$$

We can find vertex-disjoint 2-paths in $G[U_1,V_0\cup W_1]$ with all the vertices of $u\in V_0\cup W_1$ as middle vertices such that these 2-paths are all vertex-disjoint with those existing 2-paths. And $deg_{G^*}(v)\geq (1-\alpha_1-2\alpha_2)\frac{n}{2}$ for all $v\not \in W$ as before.
\end{proof}

Now we construct the desired Hamiltonian cycle in $G$. In the proof of Lemma \ref{ec1}, we know that most of  the 2-paths are greedily chosen, so we assume that $x,y$ won't be any end vertices of those 2-paths (actually in the last part of the proof of Lemma \ref{ec1}, we can also assume those moved 2-paths won't have $x,y$ as the end vertices by (\ref{s})). But $x,y$ can be the middle vertex of a 2-path.

First, assume $\frac{n}{2}$ is odd. By Lemma \ref{ec1}, we obtain a spanning bipartite graph $G^*$.

\noindent
{\em Sub-case 1:} suppose $x,y$ are in different parts of $G^*$, say $x\in U_1, y\in U_2$.

Assume $W\not =\emptyset$ and $x,y\not \in W$. We need the following claim to string all the vertices of $W$ in a path.

\begin{claim}\label{ec1c}
We can construct a path $P$ with end vertices $x_1\in U_1$ and $y_1\in U_2$ such that $P$ contains all the vertices of $W$ and $|V(P)|=4|W|$.
\end{claim}

\begin{proof}
Partition $W=W_1\cup W_2$ with $W_1=W\cap U_1$ and $W_2=W\cap U_2$. Suppose that $W_1=\{w_1,w_2,...,w_t\}$ and the two end vertices of the 2-path containing $w_i$ are $a_i,b_i$ ($1\leq i\leq t$). Since $deg_{G^*}(a_{i+1})\geq (1-\alpha_1-2\alpha_2)\frac{n}{2}$ and $deg_{G^*}(b_i)\geq (1-\alpha_1-2\alpha_2)\frac{n}{2}$ by Lemma \ref{ec1}, $a_{i+1}$ and $b_i$ have at least $(1-2\alpha_1-4\alpha_2)\frac{n}{2}$ common neighbors in $G^*$ ($1\leq i\leq t-1$). We greedily choose $c_i\in U_1$ which is a common neighbor of $a_{i+1},b_i$ ($1\leq i\leq t-1$). Since $|W|\leq \alpha_2 n$, we can choose all these $c_i$'s such that they are distinct. Let $P_1=a_1w_1b_1c_1a_2w_2b_2...c_{t-1}a_tw_tb_t$. $P_1$ contains all the vertices of $W_1$ and $|V(P_1)|=4|W_1|-1$. Similarly, we can construct another path $P_2$ which contains all the vertices of $W_2$ and $|V(P_2)|=4|W_2|-1$. Suppose the end vertices of $P_2$ are $u,v\in U_1$. We choose an unused neighbor of $v$ in $U_2$, denoted by $v^{'}$, and choose a common unused neighbor of $v^{'}, b_t$ in $U_1$, denoted by $u^{'}$. This is possible because all the vertices of $V(G^*)-V(P_1)-V(P_2)$ have degree at least $(1-\alpha_1-2\alpha_2)\frac{n}{2}-4\alpha_2 n\geq (1-\alpha_1-10\alpha_2)\frac{n}{2}$.

Let $P=P_1\cup P_2\cup \{b_tu^{'},u^{'}v^{'},v^{'}v\}$, which is the path we need. $|V(P)|=4|W|\leq 4\alpha_2n$. We denote the end vertices of $P$ by $x_1\in U_1$ and $y_1\in U_2$.
\end{proof}

Let $U_1^{*}=U_1-V(P)$ and $U_2^{*}=U_2-V(P)$. By the proof of Claim \ref{ec1c}, we can say $|U_1^{*}|=|U_2^{*}|$. For any vertex $u\in U_1^{*}$, $deg_{U_2^{*}}(u)\geq (1-\alpha_1-2\alpha_2)\frac{n}{2}-4\alpha_2n=(1-\alpha_1-10\alpha_2)\frac{n}{2}$ and for any vertex $v\in U_2^{*}$, $deg_{U_1^{*}}(v)\geq (1-\alpha_1-10\alpha_2)\frac{n}{2}$. We choose two unused neighbors of $x$ (resp. $y$), denoted by $y_2, y_3$ (resp. $x_2,x_3$), choose a common unused neighbor of $y_1,y_2$ in $U_1^*$, denoted by $x_4$, and choose an unused neighbor of $x_2$ in $U_2^*$, denoted by $y_4$. Let $U_1^{'}=(U_1^*-\{x,x_2,x_4\})\cup \{x_1\}$ and $U_2^{'}=U_2^*-\{y,y_1,y_2\}$ and $n^{'}=|U_1^{'}|=|U_2^{'}|\leq \frac{n}{2}$. For any vertex $u$ in $U_1^{'}$, $deg_{U_2^{'}}(u)\geq (1-\alpha_1-10\alpha_2)\frac{n}{2}-3$ and for any vertex $v$ in $U_2^{'}$, $deg_{U_1^{'}}(v)\geq (1-\alpha_1-10\alpha_2)\frac{n}{2}-3$. Since $n$ can be sufficiently large, we can say $(1-\alpha_1-10\alpha_2)\frac{n}{2}-3\geq (1-\alpha_1-11\alpha_2)\frac{n}{2}$. For any vertex $u$ in $U_1^{'}$, $deg_{U_2^{'}}(u)\geq (1-\alpha_1-11\alpha_2)\frac{n}{2}\geq (1-\alpha_1-11\alpha_2)n^{'}$ and for any vertex $v$ in $U_2^{'}$, $deg_{U_1^{'}}(v)\geq (1-\alpha_1-11\alpha_2)n^{'}$. By Lemma \ref{chen}, we can say $(U_1^{'},U_2^{'})$ is $(\sqrt{\alpha_1+11\alpha_2},1-\alpha_1-11\alpha_2)$-super-regular. Since $\alpha \leq (\frac{1}{9})^3$, $1-\alpha_1-11\alpha_2\geq \frac{2}{3}$. We can say $(U_1^{'},U_2^{'})$ is $(\sqrt{\alpha_1+11\alpha_2},\frac{2}{3})$-super-regular. Applying Lemma \ref{blow2} to the pair $(U_1^{'},U_2^{'})$, we construct two vertex-disjoint paths $P_1$ and $P_2$ such that the end vertices of $P_1$ are $x_1, y_4$, the end vertices of $P_2$ are $x_3, y_3$ and $|V(P_i)|$ is $l^i$ $(i=1,2)$. We denote $P_3$ to be the path $P_3:=y_1x_4y_2xy_3$ and $P_4$ to be the path $P_4:=x_3yx_2y_4$. Then
$$C=P_1\cup P_2\cup P_3\cup P_4\cup P$$
is a Hamiltonian cycle in $G$ (see Figure 4, Sub-case 1). We fix $l^1=\frac{n}{2}-1-(|V(P)|-2+4)=\frac{n}{2}-|V(P)|-3$ and $l^2=2n^{'}-l^1$. Thus it is not hard to see that $x$ and $y$ have distance $\frac{n}{2}$ on $C$.

Now assume that at least one of $x,y$ is in $W$, without loss of generality, we say $x\in W$. We can similarly construct a path $P$ with end vertices $x_1\in U_1$ and $y_1\in U_2$ such that $P$ contains all the vertices of $W-\{x,y\}$ as in Claim \ref{ec1c}. We need to make sure $P$ won't use the vertices of the 2-path which contains $x$. It is possible because we always greedily choose the vertices to construct $P$. So we can still find the unused neighbors of $x$ and finish the proof as before.

We also need to consider $W=\emptyset$. In $G^*$ we choose two neighbors of $x$, denoted by $y_1,y_2$, and two neighbors of $y$, denoted by $x_1,x_2$. Let $U_1^{'}=U_1-\{x\}$ and $U_2^{'}=U_2-\{y\}$. By Lemma \ref{chen} and Lemma \ref{ec1}, we can say $(U_1^{'},U_2^{'})$ is $(\sqrt{\alpha_1+3\alpha_2},\frac{2}{3})$-super-regular. Let $l^1=l^2=\frac{n}{2}-1$. By Lemma \ref{blow2}, we construct two vertex-disjoint paths such that the end vertices of $P_1$ are $x_1, y_1$, the end vertices of $P_2$ are $x_2, y_2$ and $|V(P_i)|$ is equal to $l^i$ $(i=1,2)$. Let $P_3:=y_1xy_2$ and $P_4:=x_1yx_2$. So $C=P_1\cup P_2\cup P_3\cup P_4$ is our desired Hamiltonian cycle.

\begin{figure}[htbp]
\centering
\includegraphics[width=13cm]{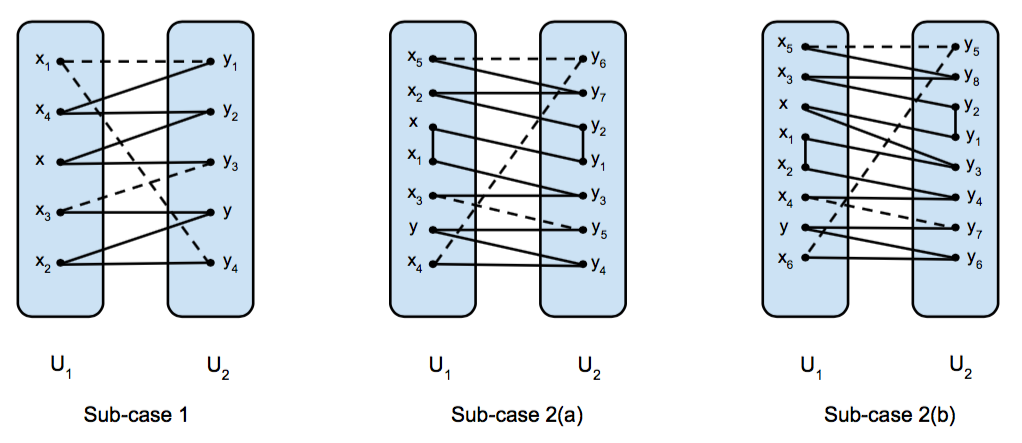}
\caption{Extremal case 1.}
\end{figure}

\noindent
{\em Sub-case 2:} suppose $x,y$ are in the same part of $G^*$, without loss of generality, say $x,y\in U_1$.

Since the construction in sub-case 1 is always in the bipartite graph $(U_1,U_2)$ and $\frac{n}{2}$ is odd, it seems that the same method doesn't work in this case. Actually we need some edges in $G[U_1]$ and $G[U_2]$ to change the parity.

Assume $W\not =\emptyset$ and $x,y\not \in W$ (if one of $x,y$ is in $W$, the discussion is almost the same as discussed before). Since $\delta (G)\geq \frac{n}{2}+1$, $x$ should have a neighbor in $U_1$.

Assume this neighbor, denoted by $x_1$, is not $y$. We choose a neighbor of $x$ in $U_2-W$, denoted by $y_1$ and choose a neighbor of $y_1$ in $U_2$, denoted by $y_2$. Whether $x_1$ is in $W$ or not, we can find an unused neighbor of $x_1$ in $U_2-W$, denoted by $y_3$, and choose an unused neighbor of $y_3$ in $U_1-W$, denoted by $x_3$. Whether $y_2$ is in $W$ or not, we can find an unused neighbor of $y_2$ in $U_1-W$, denoted by $x_2$. We choose two unused neighbors of $y$ in $U_2-W$, denoted by $y_4,y_5$, and an unused neighbor of $y_4$ in $U_1-W$, denoted by $x_4$. Since $deg_{G^*}(v)\geq (1-\alpha_1-2\alpha_2)\frac{n}{2}$ for all $v\not \in W$, it is possible to choose all these vertices as discussed in sub-case 1. By the same method of Claim \ref{ec1c}, we can construct a path $P$ with end vertices $x_5\in U_1$ and $y_6\in U_2$ such that $P$ contains all the unused vertices of $W$ and $|V(P)|\leq 4|W|$. Since the vertices used in $P$ are greedily chosen, we can assume that $P$ won't use any existing chosen vertices. We choose a common unused neighbor of $x_2,x_5$ in $U_2-W$, denoted by $y_7$. Let $U_1^{'}=U_1-V(P)-\{x,y,x_1,x_2\}$, $U_2^{'}=(U_2-V(P)-\{y_1,y_2,y_3,y_4,y_7\})\cup \{y_6\}$ and $n^{'}=|U_1^{'}|=|U_2^{'}|$. By Lemma \ref{chen} and $n$ is sufficiently large, $(U_1^{'},U_2^{'})$ is a $(\sqrt{\alpha_1+11\alpha_2},\frac{2}{3})$-super-regular pair. Applying Lemma \ref{blow2} to the pair $(U_1^{'},U_2^{'})$, we can construct two paths $P_1$ and $P_2$ such that the end vertices of $P_1$ are $x_4, y_6$, the end vertices of $P_2$ are $x_3, y_5$ and $|V(P_i)|=l^i$ $(i=1,2)$. Let $P_3:=x_5y_7x_2y_2y_1xx_1y_3x_3$ and $P_4:=x_4y_4yy_5$. Then
$$C=P_1\cup P_2\cup P_3\cup P_4\cup P$$
is a Hamiltonian cycle in $G$ (see Figure 4, Sub-case 2(a)). We fix $l^1=\frac{n}{2}-1-(|V(P)|+5-1)=\frac{n}{2}-|V(P)|-5$ and $l^2=2n^{'}-l^1$. It is not hard to see that $x$ and $y$ have distance $\frac{n}{2}$ on $C$. Here we omit all the calculations about $\alpha$, because it is almost same as sub-case 1.

Now assume that $y$ is the only neighbor of $x$ in $U_1$ but $y$ has a neighbor which is not $x$ in $U_1$, then the proof is similar to the proof in the last paragraph if we deal with $y$ first. We assume that $y$ is the only neighbor of $x$ in $U_1$ and $x$ is the only neighbor of $y$ in $U_1$. We choose a neighbor of $x$ in $U_2-W$, denoted by $y_1$. Since $deg_G(y_1)\geq \frac{n}{2}+1$, $y_1$ has a neighbor in $U_2$, denoted by $y_2$. We choose another unused neighbor of $x$ in $U_2-W$, denoted by $y_3$, and a neighbor of $y_3$ in $U_1-W$, denoted by $x_1$. Since $deg_G(x_1)\geq \frac{n}{2}+1$, $x_1$ has a neighbor in $U_1$, denoted by $x_2$. By our assumption, $x_2$ should not be either of $x$ and $y$. We choose an unused neighbor of $y_2$ in $U_1-W$, denoted by $x_3$, an unused neighbor of $x_2$ in $U_2-W$, denoted by $y_4$, and an unused neighbor of $y_4$ in $U_1-W$, denoted by $x_4$. By the same method of Claim \ref{ec1c}, we construct a path $P$ with end vertices $x_5\in U_1$ and $y_5\in U_2$ such that $P$ contains all the unused vertices of $W$ and $|V(P)|\leq 4|W|$. We choose two unused neighbors of $y$ in $U_2-V(P)$, denoted by $y_6,y_7$, and choose a neighbor of $y_6$ in $U_1-V(P)$, denoted by $x_6$, and choose a common unused neighbor of $x_3,x_5$ in $U_2-V(P)$, denoted by $y_8$. Let $U_1^{'}=U_1-V(P)-\{x,y,x_1,x_2,x_3\}$ and $U_2^{'}=(U_2-V(P)-\{y_1,y_2,y_3,y_4,y_6,y_8\})\cup \{y_5\}$ and $n^{'}=|U_1^{'}|=|U_2^{'}|$. By Lemma \ref{chen} and $n$ is sufficiently large, $(U_1^{'},U_2^{'})$ is a $(\sqrt{\alpha_1+11\alpha_2},\frac{2}{3})$-super-regular pair. Applying Lemma \ref{blow2} to the pair $(U_1^{'},U_2^{'})$, we can construct two paths $P_1$ and $P_2$ such that the end vertices of $P_1$ are $x_6, y_5$, the end vertices of $P_2$ are $x_4, y_7$ and $|V(P_i)|=l^i$ $(i=1,2)$. Let $P_3:=x_5y_8x_3y_2y_1xy_3x_1x_2y_4x_4$ and $P_4:=x_6y_6yy_7$. Then
$$C=P_1\cup P_2\cup P_3\cup P_4\cup P$$
is a Hamiltonian cycle in $G$ (see Figure 4, Sub-case 2(b)). Let $l^1=\frac{n}{2}-|V(P)|-5$ and $l^2=2n^{'}-l^1$. Thus $x$ and $y$ have distance $\frac{n}{2}$ on $C$. We can say all the choices of the vertices are possible because of the minimum degree of $G^*$. We omit all these similar calculations here.

If $W$ is empty, we take the path $P$ be an edge. The rest proof is the same as above.

At the last we need to consider the case when $\frac{n}{2}$ is even. Actually if $x,y$ are in the same part of $G^*$, the proof is similar to sub-case 1, and if $x,y$ are in the different parts of $G^*$, the proof is similar to sub-case 2.

\subsection{Extremal case 2}
Suppose $G$ is a graph on $n$ vertices with $\delta (G)\geq \frac{n}{2}+1$ and there exists a balanced partition of $V(G)$ into $V_1$ and $V_2$ such that the density $d(V_1,V_2)\leq \alpha$. We suppose $\alpha \leq (\frac{1}{9})^3$. Let $\alpha_1=\alpha ^{\frac{1}{3}}$ and $\alpha_2=\alpha ^{\frac{2}{3}}$.

We also need a similar lemma as Lemma \ref{ec1}.

\begin{lemma}\label{ec2}
If $G$ is in extremal case 2, then $V(G)$ can be partitioned into two balanced parts $U_1$ and $U_2$ such that

(a) there is a set $W_1\subseteq U_1$ (resp. $W_2\subseteq U_2$) such that there exist vertex-disjoint 2-paths in $G[U_1]$ (resp. $G[U_2]$) with the vertices of $W_1$ (resp. $W_2$) as the middle vertices in each 2-path and $|W_1|\leq \alpha_2\frac{n}{2}$ (resp. $|W_2|\leq \alpha_2\frac{n}{2}$);

(b) $deg_{G[U_1]}(u)\geq (1-\alpha_1-2\alpha_2)\frac{n}{2}$ for all $u\in U_1-W_1$ and $deg_{G[U_2]}(v)\geq (1-\alpha_1-2\alpha_2)\frac{n}{2}$ for all $v\in U_2-W_2$.
\end{lemma}

\begin{proof}
The argument is similar to the proof of Lemma \ref{ec1}. For some similar claims, we just give them without proofs.

For $i=1,2$, let $V_i^*=\{v\in V_i: deg(v,V_{i})\geq (1-\alpha_1)\frac{n}{2}\}$. We can claim that $|V_i-V_i^*|\leq \alpha_2\frac{n}{2}$ by the density condition.

For any vertex $v\in V_i-V_i^*$, if $deg(v,V_{3-i})\geq (1-\alpha_1)\frac{n}{2}$, we also add it to $V_{3-i}^*$. We denote the final two sets by $V_i^{'}$ ($i=1,2$) and let $V_0=V-V_1^{'}-V_2^{'}$. Thus $|V_0|\leq \alpha_2n$. For every vertex $v$ in $V_i^{'}$, $deg(v,V_i^{'})\geq (1-\alpha_1)\frac{n}{2}-\alpha_2\frac{n}{2}$ ($i=1,2$). For every vertex $u$ in $V_0$, $deg(u,V_i^{'})\geq (\frac{n}{2}-(1-\alpha_1)\frac{n}{2})-\alpha_2\frac{n}{2}\geq (\alpha_1-\alpha_2)\frac{n}{2}$ ($i=1,2$).

First, we assume $|V_1^{'}|,|V_2^{'}|\leq \frac{n}{2}$. We add all the vertices in $V_0$ to $V_1^{'}$ and $V_2^{'}$ such that the final two sets are of the same size. Denote the final two sets by $U_1$ and $U_2$. Let $W_1=U_1-V_1^{'}$ and $W_2=U_2-V_2^{'}$. So $V_0=W_1\cup W_2$. Since for each vertex $u\in W_1$, $deg(u,V_{1}^{'})\geq (\alpha_1-\alpha_2)\frac{n}{2}\geq 2\alpha_2n\geq 2|W|$, we can greedily choose two neighbors of $u$ in $V_1^{'}$ such that the neighbors of all the vertices in $W_1$ are distinct. So $W_1$ and $U_1$ are what we need. It is same to find 2-paths in $G[U_2]$. The degree conclusion also holds.

Second, without loss of generality we assume $|V_1^{'}|>\frac{n}{2}$. Let $V_1^0$ be the set of vertices $v\in V_1^{'}$ such that $deg(v,V_2^{'})\geq \alpha_1\frac{n}{2}$.

If $|V_1^0|\geq |V_1^{'}|-\frac{n}{2}$, we take $W_2$ to be the set of all vertices of $V_0$ and $|V_1^{'}|-\frac{n}{2}$ vertices of $V_1^0$ and $W_1$ to be an empty set. Let $U_1=V_1^{'}-W_2$ and $U_2=V_2^{'}\cup W_2$. So $|W_2|\leq \alpha_2\frac{n}{2}$. For every vertex $u\in W_2$, we have $deg(u,V_2^{'})\geq (\alpha_1-\alpha_2)\frac{n}{2}-\alpha_2\frac{n}{2}\geq \alpha_2n\geq 2|W_2|$. Thus we can greedily choose two neighbors of $u$ in $V_2^{'}$ such that the neighbors of all the vertices in $W_2$ are distinct. $U_1$, $U_2$, $W_1$, $W_2$ are what we need.

Now we assume $|V_1^0|< |V_1^{'}|-\frac{n}{2}$. Let $U_1=V_1^{'}-V_1^0$ and $U_2=V_2^{'}\cup V_0\cup V_1^0$. Let $t=|U_1|-\frac{n}{2}$, so $t\leq \alpha_2\frac{n}{2}$. We consider the bipartite graph $(U_1,V_2^{'})$. Suppose that $(U_1,V_2^{'})$ has a biggest family of vertex-disjoint 2-paths on a vertex set $S$, such that the middle vertices of these 2-paths are in $U_1$ and the end vertices of these 2-paths are in $V_2^{'}$. Let $S=S_1\cup S_2$ with the middle vertex set $S_1\subseteq U_1$ and the end vertex set $S_2\subseteq V_2^{'}$. Suppose $|S_1|=s$, $|S_2|=2s$. We use $\delta^{*}$ to denote the minimum degree of vertices of  $V_2^{'}$ in $(U_1,V_2^{'})$. So $\delta^{*}\geq \frac{n}{2}+1-(\frac{n}{2}-t-1)=t+2$. We use $\Delta^{*}$ to denote the maximum degree of vertices of $U_1$ in $(U_1,V_2^{'})$. So $\Delta^{*}< \alpha_1\frac{n}{2}$. Then $$\delta^{*}(|V_2^{'}|-2s)\leq e(V_2^{'}-S_2,U_1)\leq s(\Delta^*-2)+(\frac{n}{2}+t-s).$$

By some calculations, we can get
\begin{align*}
s &\geq \frac{(t+2)|V_2^{'}|-(\frac{n}{2}+t)}{\Delta^*+2t+1}\\
  &\geq \frac{(t+2)(1-\alpha_2)\frac{n}{2}-\frac{n}{2}-\alpha_2\frac{n}{2}}{\alpha_1\frac{n}{2}+1+2\alpha_2\frac{n}{2}} \\
  &=\frac{(t+1)(1-\alpha_2)\frac{n}{2}-2\alpha_2\frac{n}{2}}{(\alpha_1+2\alpha_2)\frac{n}{2}+1}
\end{align*}

Since $n$ can be sufficiently large, we can conclude $s>t$.

We pick $t$ vertex-disjoint 2-paths with the middle vertex set $S_1\subseteq U_1$ and move the vertices of $S_1$ into $U_2$. Now we get $|U_1|=|U_2|=\frac{n}{2}$. Let $W_2=V_0\cup V_1^0\cup S_1$ and $W_1$ be an empty set. For every vertex $u\in V_0\cup V_1^0$, $deg_{G[U_2]}(u)\geq (\alpha_1-\alpha_2)\frac{n}{2}-\alpha_2\frac{n}{2}\geq 2|W_2|$. We can greedily find disjoint 2-paths in $G[V_0\cup V_1^0,V_2^{'}]$ with all the vertices of $V_0\cup V_1^0$ as middle vertices such that these 2-paths are all disjoint with the existing 2-paths. $U_1$, $U_2$, $W_1$, $W_2$ are what we need.
\end{proof}

For a graph $G$ in extremal case 2, we apply Lemma \ref{ec2} to $G$ and get a partition of $V(G)=U_1\cup U_2$ with the properties in Lemma \ref{ec2}.

First, assume $x$ and $y$ are in different parts in the partition of $V(G)$, without loss of generality, we say that $x\in U_1$ and $y\in U_2$. Since $\delta(G)\geq \frac{n}{2}+1$, $x$ (resp. $y$) should have at least two neighbors in $U_2$ (resp. $U_1$). Denote a neighbor of $x$ in $U_2$ by $x_1$ and a neighbor of $y$ in $U_1$ by $y_1$ such that $x_1\not =y$ and $y_1\not =x$. Since the 2-paths are all greedily chosen in Lemma \ref{ec2}, we can assume that $x,y,x_1,y_1$ are not the end vertices of those 2-paths.

\begin{claim}\label{hp}
There is a Hamiltonian path in $G[U_1]$ with end vertices $x$ and $y_1$.
\end{claim}
\begin{proof}
Whether $x$ and $y_1$ are in $W_1$ or not, we can find a neighbor of $x$ in $U_1-W_1$, denoted by $u$, and a neighbor of $y_1$ in $U_1-W_1$, denoted by $v$. Suppose $W_1-\{x,y_1\}=\{w_1, w_2,...,w_t\}$ and the two end vertices of the 2-path containing $w_i$ are $a_i,b_i$. Since $deg_{G[U_1]}(a_i)\geq (1-\alpha_1-2\alpha_2)\frac{n}{2}$ and $deg_{G[U_1]}(b_i)\geq (1-\alpha_1-2\alpha_2)\frac{n}{2}$, we can greedily choose $c_i\in U_1$ which is a common neighbor of $a_{i+1}$ and $b_i$ ($1\leq i\leq t-1$). Moreover we can choose all these $c_i$ to be distinct. We also greedily choose $c_t$ which is a common neighbor of $b_t$ and $v$. Then $P_1=a_1w_1b_1c_1a_2w_2b_2c_2...b_{t-1}c_{t-1}a_{t}w_{t}b_{t}c_{t}v$ is a path containing all the vertices of those 2-paths (except the 2-paths containing $x,y_1$, if $x,y_1$ are in $W_1$). $|V(P_1)|=4t+1\leq 4\alpha_2\frac{n}{2}+1$.

Let $U^*=(U_1-V(P)-\{x,y_1\})\cup \{a_1\}$. We consider the induced subgraph $G[U^*]$. For any vertex $w\in U^*$, $deg_{G[U^*]}(w)\geq (1-\alpha_1-2\alpha_2)\frac{n}{2}-4\alpha_2 \frac{n}{2}-2$. Since $n$ is sufficiently large and $\alpha \leq (\frac{1}{9})^3$, $deg_{G[U^*]}(w)\geq (1-\alpha_1-7\alpha_2)\frac{n}{2}>\frac{n}{4}+1\geq \frac{|U^*|}{2}+1$ for any vertex $w\in U^*$. So $G[U^*]$ is Hamiltonian-connected. We can find a path $P_2$ in $G[U^*]$ with end vertices $u,a_1$ containing all the vertices of $U^*$. Then $H_1=\{xu\}\cup P_1\cup P_2\cup \{vy_1\}$ is a Hamiltonian path in $G[U_1]$ with end vertices $x$ and $y_1$.
\end{proof}

By the same method, we can construct a Hamiltonian path $H_2$ in $G[U_2]$ with end vertices $y$ and $x_1$. So
$$C=\{xx_1,yy_1\}\cup H_1\cup H_2$$
is a Hamiltonian cycle in $G$ such that $dist_C(x,y)=\frac{n}{2}$ (see Figure 5 (a)).

\begin{figure}[htbp]
\centering
\includegraphics[width=10cm]{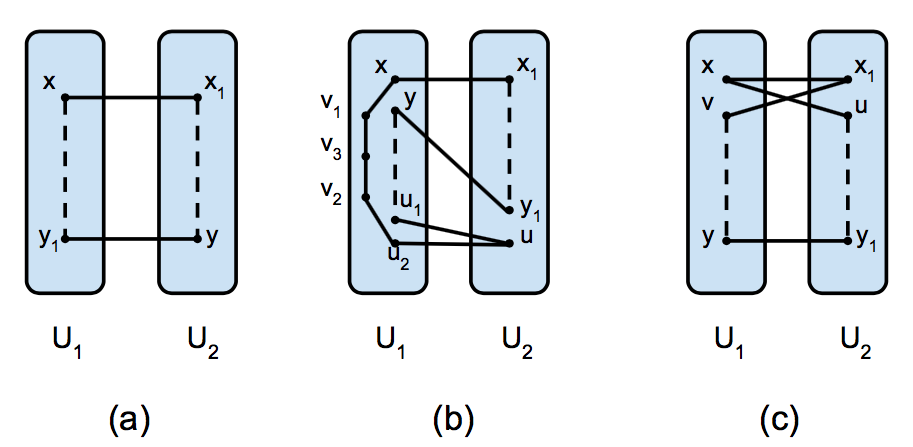}
\caption{Extremal case 2.}
\end{figure}

Now assume $x$ and $y$ are in the same part in the partition of $V(G)$, without loss of generality, say $x,y\in U_1$. Since $\delta(G)\geq \frac{n}{2}+1$, $x$ and $y$ should have at least two neighbors in $U_2$. We choose a neighbor of $x$ in $U_2$, denoted by $x_1$, and a neighbor of $y$ in $U_2$, denoted by $y_1$, such that $x_1\not =y_1$. Since the 2-paths are all greedily chosen in Lemma \ref{ec2}, we can assume that $x,y,x_1,y_1$ are not the end vertices of those 2-paths.

Assume there is a vertex $u\in U_2-\{x_1,y_1\}$ such that it has two neighbors $u_1,u_2\in U_1-\{x,y\}$. We also assume that $u_1,u_2$ are not the end vertices of the 2-paths. Whether $x,u_2$ are in $W_1$ or not, we claim that we can find a path of length at most four with end vertices $x$ and $u_2$ in $G[U_1]$. Indeed, the worst case is when $x,u_2$ are both in $W_1$. We can find a neighbor of $x$ in $U_1-W_1$, denoted by $v_1$, and a neighbor of $u_2$ in $U_1-W_1$, denoted by $v_2$. We choose a common neighbor of $v_1,v_2$ in $G[U_1]$, denoted by $v_3$. $xv_1v_3v_2u_2$ is a path of length four with end vertices $x$ and $u_2$ in $G[U_1]$. Then we can construct a path with end vertices $x_1$ and $u_1$, denote it by $P_1=x_1xv_1v_3v_2u_2uu_1$. By the same method in the proof of Claim \ref{hp}, we construct a path $P_2$ in $G[U_1]$ with end vertices $u_1$ and $y$, containing all the vertices of $U_1-V(P_1)$. In $G[U_2]$, by the same method in the proof of Claim \ref{hp}, we can construct a path $P_3$ with end vertices $x_1$ and $y_1$, containing all the vertices of $U_2-\{u\}$. So
$$C=P_1\cup P_2\cup P_3\cup \{yy_1\}$$
is a Hamiltonian cycle in $G$ such that $dist_C(x,y)=\frac{n}{2}$ (see Figure 5 (b)).

Assume there is a vertex $u\in U_2-\{x_1,y_1\}$ such that only one of the neighbors of $u$ in $U_1$ is equal to $x$ or $y$, without loss of generality, we assume that $u_1,u_2$, the two neighbors of $u$, satisfy that $u_2=x$ but $u_1\not =y$. Let $P_1=x_1xuu_1$. The rest construction is the same as in the last paragraph.

At last we assume that in $U_1$, the neighbors of all the vertices of $U_2-\{x_1,y_1\}$ are $x$ and $y$. That means any vertex in $U_1-\{x,y\}$ is adjacent to $x_1$ and $y_1$. We choose a neighbor of $x$ in $U_2$, denoted by $u$, and a neighbor of $x_1$ in $U_1$, denoted by $v$. We construct a path $P_1=vx_1xu$. By the same method in the proof of Claim \ref{hp}, we can construct a path $P_2$ in $G[U_1]$ with end vertices $v$ and $y$, containing all the vertices of $U_1-\{x\}$, and a path $P_3$ in $G[U_2]$ with end vertices $u$ and $y_1$, containing all the vertices of $U_2-\{x_1\}$. So
$$C=P_1\cup P_2\cup P_3\cup \{yy_1\}$$
is a Hamiltonian cycle in $G$ such that $dist_C(x,y)=\frac{n}{2}$ (see Figure 5 (c)).

\end{document}